\documentclass[12pt]{elsarticle}
\usepackage{amsfonts}

\usepackage{amssymb}

\newtheorem{theorem}{Theorem}

\newenvironment{proof}[1][Proof]{\noindent\textbf{#1.} }{\ \rule{0.5em}{0.5em}}
\begin{document}

\begin{frontmatter}



\title{Loewner Evolution as It\^{o} Diffusion}
\author{H\"{u}lya~Acar}
\ead{hulyaacar98@gmail.com}
\address{Department of Mathematics, Fatih University,
Istanbul, 34500, Turkey}
\author{Alexey L.~Lukashov}
\ead{alukashov@fatih.edu.tr}
\fntext[1,2]{The authors would like to thank the Scientific and Technological Research Council of Turkey (TUB{\.I}TAK) for the financial support (project no:113F369).}
\address{Department of Mathematics, Fatih University,
Istanbul, 34500, Turkey}
\address{Department of Mechanics and Mathematics,
Saratov State University, Saratov, 410012, Russia} 
\begin{abstract}

F. Bracci, M.D. Contreras, S. D\'{i}az Madrigal proved that any evolution family of order d is described by a generalized Loewner chain. G. Ivanov and A. Vasil'ev considered randomized version of the chain and found a substitution which transforms it to an It\^{o} diffusion.We generalize their result to vector randomized Loewner chain and prove there are no other possibilities to transform such Loewner chains to It\^{o} diffusions. 
\end{abstract}

\end{frontmatter}



\section{Introduction}

The Schramm-Loewner evolution (SLE), also known as a stochastic Loewner
evolution \cite{7,11} is a conformaly invariant stochastic process which
attracts many researchers during last 15 years. This process is a stochastic
generalization of the Loewner-Kufarev differential equations. SLE has the
domain Markov property which is closely related to the fact that the
equations can be represented as time homogeneous diffusion equations. There
are two main classes of SLE: Chordal SLE which gives a family of random
curves from two fixed boundary points and radial SLE which gives a family of
random curves from a fixed boundary point to a fixed interior point.

The radial equation was introduced by K. Loewner in 1923. The idea was to
represent domains by means of a family (known as Loewner chains) of
univalent functions defined on the unit disc and satisfying a suitable
differential equation.

The classical radial Loewner equation in the unit disc $\mathbb{D}$ $%
:=\left\{ \zeta \in 
\mathbb{C}
:\left\vert \zeta \right\vert <1\right\} $ is the following differential
equation

\begin{equation}  \label{1}
\left\{ 
\begin{array}{c}
\frac{d\phi (z)}{dt}=G(\phi _{t}(z),t) \\ 
\phi _{0}(z)=z%
\end{array}%
\right.
\end{equation}
for almost every $t\in \left[ 0,\infty \right) $ where $G(w,t)=-wp(w,t)$
with the function $p:\mathbb{D\times }\left[ 0,\infty \right) \rightarrow 
\mathbb{C} $ measurable in $t$, holomorphic in $z$, $p(0,t)=1$ and $%
\Re(p(z,t))\geq0$ for all $z\in\mathbb{D}$ and $t\geq0$ (such functions $p$
are called Hergl\"{o}tz functions).

Recently Georgy Ivanov and Alexander Vasil'ev \cite{4} considered random
version of this Loewner differential equation with $G(w,t)=\frac{(\tau
(t)-w)^{2}p(w,t)}{\tau (t)}$ for

\begin{equation}  \label{2}
\tau (t)=\tau (t,w)=\exp (ikB_{t}(w)).
\end{equation}
They found a substitution which transforms the randomized Loewner equation
with $p(w,t)=\tilde{p}\left( \frac{w}{\tau (t)}\right) $ to an It\^{o}
diffusion and obtained the infinitesimal generator of the It\^{o} diffusion
in this form: 
\begin{equation}  \label{3}
A=\left( -\frac{z}{2}k^{2}+(1-z)^{2}\tilde{p}(z)\right) \frac{d}{dz}-\frac{1 
}{2}k^{2}z^{2}\frac{d^{2}}{dz^{2}}.
\end{equation}

The main result is an inverse statement. Namely we prove that under rather
general suppositions on $\tau (t)=\tau (t,B_{t})$, it is possible to find a
substitution which transforms (\ref{1}) to an It\^{o} diffusion if and only
if $\tau $ is given by (\ref{2}). We generalize this necessary and
sufficient condition for higher dimensions when $\tau $ depends on some
independent Brownian motions 
\[
\tau (t)=\tau (\mathbf{B}_{t}) 
\]
where $\mathbf{B}_{t}=(B_{t}^{1},B_{t}^{2},\ldots,B_{t}^{n}).$

We denote by $\check{C}$ the set of functions $f(z,\mathbf{x })$ from $C^{n}(%
\mathbb{D} \times \mathbb{R}^{n})$ such that these functions have continuous
derivatives up to order $n$, $\frac{\partial f}{\partial z}$ doesn't vanish
and $H(\mathbb{D})$ is the set of analytic functions in $\mathbb{D}$. 
\newline
\newline

\begin{theorem}
Consider Loewner random differential equation

\begin{equation}  \label{4}
\left\{ 
\begin{array}{c}
\frac{d\phi _{t}\left( z,w\right) }{dt}=\frac{\left( \tau _{1}\left(
t,w\right) -\phi _{t}\left( z,w\right) \right) ^{2}}{\tau _{1}\left(
t,w\right) }\tilde{p}(\frac{\phi _{t}\left( z,w\right) }{\tau _{1}\left(
t,w\right) }) \\ 
\phi _{0}\left( z,w\right) =z%
\end{array}%
\right.  \label{lee}
\end{equation}
where $\left\vert \tau _{1}\left( t,\omega \right) \right\vert =1$ for each
fixed $w\in \Omega $ ($\Omega $ is a sample space) and $\tilde{p}$ is an
arbitrary Hergl\"{o}tz function. Suppose $\psi _{t}=m(\phi
_{t},B_{t}^{(1)},B_{t}^{(2)},\ldots,B_{t}^{(n)})$ where $B_{t}^{(i)}$ are
independent Brownian motions, $m\in \breve{C}$ \ and $\tau _{1}\left(
t,\omega \right) =\tau (\mathbf{B}_{t})$ then, $\psi _{t}$ is an $n\times 1$
dimensional It\^{o} diffusion with coefficients from $H(\mathbb{D})$ for an
arbitrary Hergl\"{o}tz function $\tilde{p}$ if and only if $\tau (\mathbf{B}%
_{t})=e^{\mathbf{k}\cdot \mathbf{B}_{t}}$ where $\mathbf{k=(}%
k_{1},\ldots,k_{n}\mathbf{)}$ and $\mathbf{k\in \mathbb{R}}^{n}$.

Furthermore the infinitesimal generator of $\psi _{t}$ (when it is an It\^{o}
diffusion) is given by this form 
\begin{equation}  \label{5}
A=\left( -\frac{z}{2}\left\vert \mathbf{k}\right\vert ^{2}+(1-z)^{2}\tilde{p}%
(z)\right) \frac{d}{dz}-\frac{1}{2}\left\vert \mathbf{k}\right\vert ^{2}z^{2}%
\frac{d^{2}}{dz^{2}}.
\end{equation}
\end{theorem}

\begin{proof}
For $n=1$ sufficiency part was proved by G. Ivanov and A. Vasilev. We use
similar argument to prove sufficiency for arbitrary $n$. By the complex It%
\^{o} formula, the process 
\[
\frac{1}{\tau (\mathbf{B}_{t})}=e^{-ik_{1}B_{t}^{(1)}-ik_{2}B_{t}^{(2)}-%
\ldots-ik_{n}B_{t}^{(n)}} 
\]%
\ satisfies the stochastic differential equation (SDE)

\begin{equation}  \label{6}
\begin{array}{c}
d(e^{-ik_{1}B_{t}^{(1)}-\ldots-ik_{n}B_{t}^{(n)}})=-\sum%
\limits_{j=1}^{n}ik_{j}e^{-ik_{1}B_{t}^{(1)}-%
\ldots-ik_{n}B_{t}^{(n)}}dB_{t}^{(j)} \\ 
\\ 
-\frac{1}{2}\sum\limits_{j=1}^{n}k_{j}^{2}e^{-ik_{1}B_{t}^{(1)}-%
\ldots-ik_{n}B_{t}^{(n)}}dt.%
\end{array}%
\end{equation}

Let us denote $\psi _{t}(z,w)=\frac{\phi _{t}(z,w)}{\tau (\mathbf{B}_{t})}$.
Applying the integration by parts formula for\ $\psi _{t}$, we obtain

\begin{equation}  \label{7}
\begin{array}{c}
d(\psi _{t})=\phi
_{t}d(e^{-ik_{1}B_{t}^{(1)}-%
\ldots-ik_{n}B_{t}^{(n)}})+(e^{-ik_{1}B_{t}^{(1)}-%
\ldots-ik_{n}B_{t}^{(n)}})d\phi _{t} \\ 
\\ 
=e^{-ik_{1}B_{t}^{(1)}-\ldots-ik_{n}B_{t}^{(n)}}\frac{\left(
e^{ik_{1}B_{t}^{(1)}+\ldots+ik_{n}B_{t}^{(n)}}-\phi _{t}\left( z,w\right)
\right) ^{2}}{e^{ik_{1}B_{t}^{(1)}+\ldots+ik_{n}B_{t}^{(n)}}}\tilde{p}(\psi
_{t})dt \\ 
\\ 
-i\psi _{t}\sum\limits_{j=1}^{n}k_{j}dB_{t}^{(j)}-\frac{\psi _{t}}{2}%
\sum\limits_{j=1}^{n}k_{j}^{2}dt \\ 
\\ 
=-i\psi _{t}\mathbf{k}\cdot d\mathbf{B}_{t}+(-\frac{\left\vert \mathbf{k}%
\right\vert ^{2}}{2}\psi _{t}+(\psi _{t}-1)^{2}\tilde{p}(\psi _{t}))dt%
\end{array}%
\end{equation}

So $\psi _{t}$ is an It\^{o} diffusion in $%
\mathbb{R}
^{n}$.

Now for the necessity part, from our supposition $\psi _{t}=m(\phi
_{t},B_{t}^{(1)},\ldots,B_{t}^{(n)})$. Apply It\^{o} formula;

\begin{equation}  \label{8}
d(\psi _{t})=\frac{\partial m}{\partial x}d\phi _{t}+\sum\limits_{i=1}^{n}%
\frac{\partial m}{\partial y_{i}}dB_{t}^{(i)}+\frac{1}{2}\sum%
\limits_{i=1}^{n}\frac{\partial ^{2}m}{\partial y_{i}^{2}}dt.
\end{equation}
and if (\ref{8}) is an $n\times 1$ dimensional It\^{o} diffusion with
analytic coefficients then there are functions $f_{i}\in H(\mathbb{D})$ such
that

\begin{equation}  \label{9}
\frac{\partial m}{\partial y_{i}}=f_{i}(m(x,\mathbf{y})).
\end{equation}
Taking derivative of (\ref{9}) with respect to $y_{j}$ we obtain

\[
\frac{f_{i}^{\prime }(m(x,\mathbf{y}))}{f_{i}(m(x,\mathbf{y}))}=\frac{%
f_{j}^{\prime }(m(x,\mathbf{y}))}{f_{j}(m(x,\mathbf{y}))}.
\]
Hence 
\[
(\ln f_{i}(z))^{\prime }=(\ln f_{j}(z))^{\prime } 
\]
and 
\begin{equation}  \label{10}
f_{i}(z)=c_{ij}f_{j}(z).
\end{equation}
Let us denote 
\begin{equation}  \label{11}
f_{i}(z)=c_{i}f(z)
\end{equation}
and let $F(z)$ be an antiderivative of $\frac{1}{f(z)}$, hence%
\[
F(m(x,\mathbf{y}))=\mathbf{c}\cdot \mathbf{y}+q(x) 
\]
where $\mathbf{c=(}c_{1},\ldots,c_{n}\mathbf{)}$.

Since by supposition $F^{\prime }$ doesn't vanish, there exists an inverse
function $F^{-1}$ and 
\begin{equation}  \label{12}
m(x,\mathbf{y})=F^{-1}(\mathbf{c}\cdot \mathbf{y}+q(x)).
\end{equation}
Let us denote 
\begin{equation}  \label{13}
F^{-1}(z)=G(z).
\end{equation}
Now for coefficients in $dt$, we have%
\[
\frac{\partial m}{\partial x}d\Phi _{t}+\frac{1}{2}\sum\limits_{i=1}^{n}%
\frac{\partial ^{2}m}{\partial y_{i}^{2}}dt=g(G(\mathbf{c}\cdot \mathbf{y}%
+q(x))dt. 
\]
where $g$ is an analytic function in $\mathbb{D}$. If we substitute (\ref{9}%
, \ref{11}, \ref{13}) then we get, 
\begin{equation}  \label{14}
\begin{array}{c}
G^{\prime }(\mathbf{c}\cdot \mathbf{y}+q(x))q^{\prime }(x)\frac{d\phi _{t}}{%
dt}+f^{\prime }(G(\mathbf{c}\cdot \mathbf{y}+q(x)))f(G(\mathbf{c}\cdot 
\mathbf{y}+q(x)))\frac{1}{2}\sum\limits_{i=1}^{n}c_{i}^{2} \\ 
\\ 
=g(G(\mathbf{c}\cdot \mathbf{y}+q(x))).%
\end{array}%
\end{equation}

\begin{equation}  \label{15}
G^{\prime }(\mathbf{c}\cdot \mathbf{y}+q(x))q^{\prime }(x)\frac{(\tau (%
\mathbf{y})-x)^{2}}{\tau (\mathbf{y})}\tilde{p}(\frac{x}{\tau (\mathbf{y})}%
)=g_{1}(\mathbf{c}\cdot \mathbf{y}+q(x))
\end{equation}
where we denote $%
\begin{array}{c}
\\ 
g_{1}(z)=g(z)-f^{\prime }(z)f(z)\frac{1}{2}\sum\limits_{i=1}^{n}c_{i}^{2} .
\\ 
\end{array}%
$

Definition of (\ref{12}) shows that $G^{\prime }(\mathbf{c}\cdot \mathbf{y}%
+q(x))$ doesn't vanish. Hence $g_{1}(\mathbf{c}\cdot \mathbf{y}+q(x))$ is
not identically zero. So equation (\ref{15}) can be written as

\begin{equation}  \label{16}
q^{\prime }(x)\frac{(\tau (\mathbf{y})-x)^{2}}{\tau (\mathbf{y})}\tilde{p}(%
\frac{x}{\tau (\mathbf{y})})=H_{1}(\mathbf{c}\cdot \mathbf{y}+q(x))
\end{equation}
where $H_{1}(z)=\frac{g_{1}(z)}{G^{\prime }(z)}$.

Let us differentiate it with respect to $x$ and $y_{i}$. Then we obtain two
equalities:

\begin{equation}  \label{17}
\begin{array}{c}
q^{\prime \prime }(x)\frac{(\tau (\mathbf{y})-x)^{2}}{\tau (\mathbf{y})}%
\tilde{p}(\frac{x}{\tau (\mathbf{y})})-2q^{\prime }(x)\frac{(\tau (\mathbf{y}%
)-x)}{\tau (\mathbf{y})}\tilde{p}(\frac{x}{\tau (\mathbf{y})})+q^{\prime }(x)%
\frac{(\tau (\mathbf{y})-x)^{2}}{\tau ^{2}(\mathbf{y})}\tilde{p}^{\prime }(%
\frac{x}{\tau (\mathbf{y})}) \\ 
\\ 
=q^{\prime }(x)H_{1}^{\prime }(\mathbf{c}\cdot \mathbf{y}+q(x))%
\end{array}%
\end{equation}
and 
\begin{equation}  \label{18}
\begin{array}{c}
q^{\prime }(x)\frac{(\tau (\mathbf{y})-x)\frac{\partial \tau (\mathbf{y})}{%
\partial y_{i}}(\tau (\mathbf{y})+x)}{\tau ^{2}(\mathbf{y})}\tilde{p}(\frac{x%
}{\tau (\mathbf{y})})-q^{\prime }(x)\frac{(\tau (\mathbf{y})-x)^{2}\frac{%
\partial \tau (\mathbf{y})}{\partial y_{i}}x}{\tau ^{2}(\mathbf{y})}\tilde{p}%
^{\prime }(\frac{x}{\tau (\mathbf{y})}) \\ 
\\ 
=c_{i}H_{1}^{\prime }(\mathbf{c}\cdot \mathbf{y}+q(x)); \ \ i=1,2,\ldots,n.%
\end{array}%
\end{equation}
Now (\ref{17}) and (\ref{18}) imply,

\begin{equation}  \label{19}
\begin{array}{c}
q^{\prime \prime }(x)\frac{(\tau (\mathbf{y})-x)^{2}}{\tau (\mathbf{y})}%
\tilde{p}\left( \frac{x}{\tau (\mathbf{y})}\right) -2q^{\prime }(x)\frac{%
(\tau (\mathbf{y})-x)}{\tau (\mathbf{y})}\tilde{p}\left( \frac{x}{\tau (%
\mathbf{y})}\right) +q^{\prime }(x)\frac{(\tau (\mathbf{y})-x)^{2}}{\tau
^{2}(\mathbf{y})}\tilde{p}^{\prime }\left( \frac{x}{\tau (\mathbf{y})}\right)
\\ 
\\ 
=\frac{q^{\prime }(x)}{c_{i}}\left[ q^{\prime }(x)\frac{(\tau (\mathbf{y})-x)%
\frac{\partial \tau (\mathbf{y})}{\partial y_{i}}(\tau (\mathbf{y})+x)}{\tau
^{2}(\mathbf{y})}\tilde{p}\left( \frac{x}{\tau (\mathbf{y})}\right)
-q^{\prime }(x)\frac{(\tau (\mathbf{y})-x)^{2}\frac{\partial \tau (\mathbf{y}%
)}{\partial y_{i}}x}{\tau ^{2}(\mathbf{y})}\tilde{p}^{\prime }\left( \frac{x%
}{\tau (\mathbf{y})}\right) \right]; \\ 
i=1,2,\ldots,n.%
\end{array}%
\end{equation}
Take derivative of (\ref{19}) with respect to $y_{i}$ again ; 
\begin{equation}  \label{20}
\begin{array}{c}
\left[ \tilde{p}^{\prime }\left( \frac{x}{\tau (\mathbf{y})}\right) \frac{x%
\frac{\partial \tau (\mathbf{y})}{\partial y_{i}}}{\tau ^{2}(\mathbf{y})}%
(-q^{\prime \prime }(x)(\tau (\mathbf{y})-x)+3q^{\prime }(x))-\frac{%
q^{\prime }(x)x\frac{\partial \tau (\mathbf{y})}{\partial y_{i}}(\tau (%
\mathbf{y})-x)}{\tau ^{3}(\mathbf{y})}\tilde{p}^{\prime \prime }\left( \frac{%
x}{\tau (\mathbf{y})}\right) \right. \\ 
\\ 
\left. +q^{\prime \prime }(x)\frac{\partial \tau (\mathbf{y})}{\partial y_{i}%
}\tilde{p}\left( \frac{x}{\tau (\mathbf{y})}\right) \right] \cdot \left[
q^{\prime }(x)\frac{\frac{\partial \tau (\mathbf{y})}{\partial y_{i}}(\tau (%
\mathbf{y})+x)}{\tau (\mathbf{y})}\tilde{p}\left( \frac{x}{\tau (\mathbf{y})}%
\right) -q^{\prime }(x)\frac{(\tau (\mathbf{y})-x)\frac{\partial \tau (%
\mathbf{y})}{\partial y_{i}}x}{\tau ^{2}(\mathbf{y})}\tilde{p}^{\prime
}\left( \frac{x}{\tau (\mathbf{y})}\right) \right] \\ 
\\ 
=\left[ q^{\prime }(x)\frac{\frac{\partial ^{2}\tau (\mathbf{y})}{\partial
y_{i^{2}}}\tau (\mathbf{y})(\tau (\mathbf{y})+x)-(\frac{\partial \tau (%
\mathbf{y})}{\partial y_{i}})^{2}x}{\tau ^{2}(\mathbf{y})}\tilde{p}\left( 
\frac{x}{\tau (\mathbf{y})}\right) -\frac{q^{\prime }(x)}{\tau ^{3}(\mathbf{y%
})}(x\frac{\partial ^{2}\tau (\mathbf{y})}{\partial y_{i^{2}}}\tau (\mathbf{y%
})(\tau (\mathbf{y})-x)\right. \\ 
\\ 
\left. +3x^{2}(\frac{\partial \tau (\mathbf{y})}{\partial y_{i}})^{2})\tilde{%
p}^{\prime }\left( \frac{x}{\tau (\mathbf{y})}\right) +q^{\prime }(x)\frac{%
(\tau (\mathbf{y})-x))x^{2}(\frac{\partial \tau (\mathbf{y})}{\partial y_{i}}%
)^{2}}{\tau ^{4}(\mathbf{y})}\tilde{p}^{\prime \prime }\left( \frac{x}{\tau (%
\mathbf{y})}\right) \right] \cdot \left[ q^{\prime \prime }(x)(\tau (\mathbf{%
y})-x)\tilde{p}\left( \frac{x}{\tau (\mathbf{y})}\right) \right. \\ 
\\ 
\left. -2q^{\prime }(x)\tilde{p}\left( \frac{x}{\tau (\mathbf{y})}\right)
+q^{\prime }(x)(\tau (\mathbf{y})-x)\tilde{p}^{\prime }\left( \frac{x}{\tau (%
\mathbf{y})}\right) \right]; \ \ i=1,2,\ldots,n. \\ 
\end{array}%
\end{equation}

Observe that the functions $(\tilde{p}^{\prime }\left( z\right) )^{2}$, $%
\tilde{p}^{\prime }\left( z\right) \tilde{p}\left( z\right) $, $\left( 
\tilde{p}\left( z\right) \right) ^{2}$, $\tilde{p}^{\prime \prime }\left(
z\right) \tilde{p}\left( z\right) $, $\tilde{p}^{\prime \prime }\left(
z\right) \tilde{p}^{\prime }\left( z\right)$ and $\left( \tilde{p}^{\prime
\prime }\left( z\right) \right) ^{2}$,where $\tilde{p}$ are arbitrary Hergl%
\"{o}tz functions, are independent. In fact they are independent even for $%
\tilde{p}(w)=\frac{1}{1-w}+a,a>0$, what can be checked by straightforward
calculations. Hence coefficients in $(\tilde{p}^{\prime }\left( z\right)
)^{2}$, $\tilde{p}^{\prime }\left( z\right) \tilde{p}\left( z\right) $, $%
\left( \tilde{p}\left( z\right) \right) ^{2}$, $\tilde{p}^{\prime \prime
}\left( z\right) \tilde{p}\left( z\right) $, $\tilde{p}^{\prime \prime
}\left( z\right) \tilde{p}^{\prime }\left( z\right) $and $\left( \tilde{p}%
^{\prime \prime }\left( z\right) \right) ^{2}$ in the left and right hand
part of (\ref{20})coincide. \newline

In particular,

\begin{equation}  \label{21}
x\left( \frac{\partial \tau (\mathbf{y})}{\partial y_{i}}\right)
^{2}q^{\prime \prime }(x)=-\tau (\mathbf{y})\frac{\partial ^{2}\tau (\mathbf{%
y})}{\partial y_{i}^{2}}q^{\prime }(x)
\end{equation}
and 
\begin{equation}  \label{22}
q^{\prime }(x)=-xq^{\prime \prime }(x).
\end{equation}
Then 
\begin{equation}  \label{23}
q(x)=\alpha \ln x,
\end{equation}
where $\alpha$ is a constant. If we substitute this in (\ref{21}) then we
obtain

\begin{equation}  \label{24}
\left( \frac{\partial \tau (\mathbf{y})}{\partial y_{i}}\right) ^{2}=\tau (%
\mathbf{y})\frac{\partial ^{2}\tau (\mathbf{y})}{\partial y_{i}^{2}}, \ \
1\leq i\leq n.
\end{equation}
It is easy to see that any solution of system (\ref{24}) can be written as

\[
\begin{array}{c}
\\ 
\tau (\mathbf{y})=h_{1}(y_{2},\ldots ,y_{n})e^{y_{1}g_{1}(y_{2},\ldots
,y_{n})}=h_{2}(y_{1},y_{3},\ldots ,y_{n})e^{y_{2}g_{2}(y_{1},y_{3},\ldots
,y_{n})}= \\ 
\\ 
\cdot \cdot \cdot =h_{n}(y_{1},\ldots ,y_{n-1})e^{y_{n}g_{n}(y_{1},\ldots
,y_{n-1})}, \\ 
\end{array}%
\newline
\]%
\newline
where $h_{1},\ldots ,h_{n}$ are sufficiently smooth functions.Then%
\begin{equation}
\frac{\partial ^{n}\ln \tau (\mathbf{y})}{\partial y_{1}\ldots \partial y_{n}%
}=\frac{\partial ^{n-1}g_{1}(y_{2},\ldots ,y_{n})}{\partial y_{2}\ldots
\partial y_{n}}=\cdot \cdot \cdot =\frac{\partial ^{n-1}\partial
g_{n}(y_{1},\ldots ,y_{n-1})}{\partial y_{1}\ldots \partial y_{n-1}}=c.
\label{25}
\end{equation}%
It gives the general solution of (\ref{24}) 
\begin{equation}
\ln \tau (\mathbf{y})=cy_{n}\ldots y_{1}+\tilde{g}_{1}(y_{2},\ldots
,y_{n})+\ldots +\tilde{g}_{n}(y_{1},\ldots ,y_{n-1}),  \label{26}
\end{equation}%
where $\tilde{g}_{1},\ldots ,\tilde{g}_{n}$ are arbitrary sufficiently
smooth functions.If we put this $\tau (\mathbf{y})$ in (\ref{19}) and take
coefficient of $\tilde{p}^{\prime }\left( z\right) $, then we obtain

\begin{equation}  \label{27}
\frac{c_{i}}{t_{i}}=\prod\limits_{k=1}^{n}y_{k}+\sum\limits_{k=1,k\neq i}^{n}%
\frac{\partial \tilde{g}_{k}(y_{1},...,\tilde{y}_{i},...,y_{n})}{\partial
y_{i}}
\end{equation}

Taking derivative $\frac{\partial ^{n-1}}{\partial y_{2}...\partial y_{n}}$
of (\ref{27}) with $i=1$, we obtain $c=0$.

Moreover we claim that (\ref{26}) and (\ref{27}) imply $\tau (\mathbf{y}%
)=\exp (\mathbf{K}\cdot \mathbf{y})$. Indeed for $n=1$ it follows
immediately from (\ref{24}). For $n>1$ we take derivative of (\ref{27}) with
respect to $y_{2},\ldots,y_{n-1}$ we obtain 
\[
\frac{\partial ^{n-1}\tilde{g}_{n}(y_{1},\ldots,y_{n-1})}{\partial
y_{1}\ldots\partial y_{n-1}}=0. 
\]

Hence $\tilde{g}_{n}$ can be written as sum of functions of $n-2$ variables.
By induction it implies $\tau (\mathbf{y})=\exp
(\sum\limits_{i=1}^{n}K_{i}(y_{i}))$, and application of (\ref{24}) finishes
the proof of the necessity part.

Now $[7,$ Theorem $7.3.3]$ says that the generator $A$ of the process $\psi
_{t}$ from (\ref{7}) is given by (\ref{5}).
\end{proof}

\end{document}